\documentclass[11pt,reqno]{amsart}
\usepackage{amsmath}
\usepackage{amssymb}
\usepackage{mhequ}
\usepackage{color}
\usepackage{graphicx}
\usepackage{subfigure}
\usepackage{enumerate}
\usepackage[unicode=true,
 bookmarks=true,bookmarksnumbered=false,bookmarksopen=false,
 breaklinks=false,pdfborder={0 0 1},backref=false,colorlinks=true]{hyperref}

\newtheorem{theorem}{Theorem}[section] 
\newtheorem{corollary}[theorem]{Corollary}
\newtheorem{proposition}[theorem]{Proposition}

\newtheorem{definition}{Definition}

\newtheorem{remark}[theorem]{Remark}

\newtheorem{example}{Example}

\newcommand{\N} {\mathbb N}
\newcommand{\R} {\mathbb R}
\newcommand{\Q} {\mathbb Q}

\newcommand{\Z} {\mathbb Z}

\newcommand{\Prob}{\mathbb P}
\newcommand{\E}{\mathbb E}


\begin{document}

\title[Rotation numbers for random dynamics on the circle]
 {A family of rotation numbers for discrete random dynamics on the circle}

\author{Christian S. Rodrigues}
\address{Christian S.~Rodrigues\\
   Max-Planck-Institute for Mathematics in the Sciences\\
   Inselstr. 22\\
   04103 Leipzig\\
   Germany}
   \email{christian.rodrigues@mis.mpg.de}

\author{Paulo R. C. Ruffino}
\address{Paulo R. C.~Ruffino\\
   Departamento de Matem\'{a}tica, Universidade Estadual de Campinas\\
   13.083-859 Campinas - SP\\
   Brazil}
   \email{ruffino@ime.unicamp.br}

\date{\today}

\begin{abstract}
 
  We revisit the problem of well-defining rotation numbers for
  discrete random dynamical systems on $S^1$. We show that,
  contrasting with deterministic systems, the topological (i.e. based
  on Poincar\'{e} lifts) approach does depend on the choice of lifts
  (e.g. continuously for nonatomic randomness). Furthermore, the
  winding orbit rotation number does not agree with the topological
  rotation number. Existence and conversion formulae between these
  distinct numbers are presented. Finally, we prove a sampling in time
  theorem which recovers the rotation number of continuous Stratonovich
  stochastic dynamical systems on $S^1$ out of its time discretisation
  of the flow.

\end{abstract}

\maketitle

\section{Introduction}

Rotation number is one of the most fundamental quantities
characterising the behaviour of dynamics on the circle. Since its
introduction by Poincar\'{e} more than a century ago, it has played
essential role on the understanding of the iterated deterministic
dynamics on $S^{1}$: from the existence of periodic orbits to proving
linear conjugacies. Its importance stems from the fact that it is an
invariant under conjugacy, allowing for classification of possible
asymptotic dynamics; we refer to Katok and Hasselblatt \cite{Katok and
  Hasselblatt}, and references therein for the classical
definitions. Among many other basic properties, the rotation number of
deterministic dynamics equals the average speed of orbits winding
around the circle $\pmod{ 2 \pi}$. Moreover, it is intrinsic in the sense
that it is independent: of the choice of the lift, of the initial
condition for the lifts, and of the orbit chosen to count the
winding. These well known properties are indispensable in order to
obtain simultaneous linear conjugacy of commuting diffeomorphisms to
pure rotations \cite{Fayad-Khanin}, to establish that the group of
orientation preserving homeomorphisms of the circle is a simple group
where each of its element is generated by a product of (at most three)
involutions \cite{Gill and O'Farrel}, or, as it has been recently
accomplished, to investigate general rigidity conditions of critical
maps \cite{GuM13}, just to cite a few instances.

Contrasting with deterministic dynamics, for random dynamics on the
circle many of these properties actually do not hold. Rational
rotation number, for example, does not imply the existence of periodic
orbit. Neither does irrationality of the rotation number (and
regularity) imply dense orbit. One may think of simple examples of
random systems with irrational rotation number and finite orbits:
consider for instance a random dynamics generated by i.i.d. random
mappings in $\{I, R_{\theta} \}$ where $I$ is the identity and
$R_{\theta}$ is a pure rotation by an angle $2\pi \theta$ with
rational $\theta$. Suppose that $R_{\theta}$ has probability $p\in
[0,1] \setminus \Q$ and $I$ has probability $1-p$. Normalising the
angles by $2 \pi$, any reasonable definition of rotation number leads,
in this case, to an irrational averaging rotation by $\theta p \pmod{
  2 \pi} $, with finite orbits.

We are going to show that for random dynamics of homeomorphism
preserving orientation on the circle, the topological concept based on
Poincar\'{e} lifts for each random mapping and the orbit winding
rotation concept even lead to different numbers. Therefore, they
measure distinct features of the dynamics. Yet, there are several
reasons why one is interested in such numbers. In this paper, we study
these different concepts, proving existence, investigating ergodic
properties, and the relation among them. We introduce a convenient
parametrisation of the topological rotation number, denoted here by
$\rho_{q, \alpha}$, and the orbit winding rotation number, starting at
$s_{0} \in S^{1}$, denoted here by $OR_{s_0}$. The topological
rotation number depends precisely on a choice of such parameters $(q,
\alpha)$, which encodes the dependence on the lifts for each random
homeomorphism.

Amongst many natural motivations to deal with rotation numbers for
discrete random dynamical systems on the circle, consider, for
example, random homeomorphisms on $S^1$ induced by the action of
random $2\times 2$-matrices. The rotation number for this random
system on $S^1$ is the angular counterpart of the Lyapunov exponents
for the product of these random matrices. In this sense, for a product
of random (ergodically generated) matrices, the rotation number in
invariant 2-subspaces represents the imaginary part of a generalised
eigenvalues, whose real part is represented by the Lyapunov exponents,
described in the multiplicative ergodic theorem (see e.g. Arnold
\cite{Arnold} and references therein).

Another motivation comes from discretisation in time of continuous
systems. In general, one would like to recover the original rotation
number based on observations of the system at discrete time, sampled
at intervals of length $\Delta t$. For deterministic systems, the
original rotation number can be recovered after renormalisation, if
the interval $\Delta t$ is smaller than $\frac{1}{2 \rho}$, where
$\rho$ is the rotation number (frequency). This is the classical
Nyquist frequency sampling theorem, see e.g. Higgins \cite{Higgins},
Oppenheim and Schafer \cite{Oppenheim and Schafer}.  We show here that
for appropriate parameters $(q,\alpha)$, a sampling theorem for
stochastic systems holds, as a limit when $\Delta t$ goes to zero, if
one considers the topological rotation number $\rho_{q, \alpha}$.

Several authors, from different perspectives, have addressed random
dynamics on the circle. Just to mention a few, Ruffino has proved a
sampling theorem for linear random systems on $S^1$
\cite{Ruffino_ROSE}. In \cite{Zmarrou and Homburg-2007, Zmarrou and
  Homburg-2008} Zmarrou and Homburg studied the influence of bounded
noise on bifurcation of diffeomorphism, while Li and Lu \cite{Li and
  Lu} have proved the continuity of the rotation number with respect
to $L^1$-norm on the lifts. Numerical simulations have also been
performed in McSharry and Ruffino \cite{McSharry and Ruffino}. Our
contribution here on clarifying the distinction among different
rotation numbers is not only to prevent ambiguities, but also to show
how rich the angular asymptotic behaviour is on random systems on
$S^1$.
The paper is organised as follows. In the next Section we introduce
the random dynamics based on a probability space $(\Omega,
\mathcal{F}, \Prob)$. Under \textit{uniform conditions} on the random
lifts we introduce the definition of the topological rotation number
$\rho_{q, \alpha}$. We prove an ergodic result of existence
$\Prob$-a.s. for $\rho_{q, \alpha}$. In Section 3, we introduce the
orbital rotation number. An ergodic result on existence also holds in
this case. Here, remarkably, it appears the dependence on the initial
condition, \textit{i.e.} the dependence on the ergodic invariant
measure chosen in the domain of the skew product on $\Omega \times
S^1$ (or just on $S^1$ if the system is \textit{i.i.d.}). Formulae
which compare each one of these distinct rotation numbers are
presented. Finally, in Section 4, we prove a sampling in time theorem
to recover the rotation number of continuous Stratonovich stochastic
dynamical systems on $S^1$ out of its time discretisation of the flow.


\section{Topological rotation numbers $\rho_{q, \alpha}$}

Let $(\Omega, \mathcal{F}, \Prob)$ be a probability space and $f:
\Omega \times S^{1} \rightarrow S^{1}$ be a random variable in
$\mathcal{H}^+$, the space of homeomorphisms of $S^{1}$ which preserve
the orientation. Let also $\theta: \Omega \rightarrow \Omega$ be an
ergodic transformation with respect to $\Prob$. We consider the
dynamics given by the cocycle generated by the composition of the
sequence of random homeomorphisms in the following sense (see
e.g. Arnold \cite{Arnold}),
\[
f^n(\omega, s_0)= f(\theta^{n-1}\omega, \cdot) \circ \cdots \circ
f(\theta \omega, \cdot) \circ f(\omega, s_0), \mbox{ for } n\in \N.
\]

\subsection*{Lifts of random homeomorphims.} We shall denote 
the covering map of $S^1$ by
\begin{displaymath}
\begin{split}
p : & \ \R \longrightarrow S^{1},\\
&\ x\longmapsto e^{i 2 \pi x}.
\end{split}
\end{displaymath}
A lift of an orientation preserving homeomorphism $f: S^1 \rightarrow
S^1$ is an increasing continuous function on the covering space $F :
\R \rightarrow \R$ which is semiconjugate to $f$ by $p$,
\textit{i.e.}, $p \circ F = f\circ p$.  Given a lift $F$ of $f$, then
$F+k$ is also a lift for any $k \in \Z$. For a fixed lift $F$, one can
write $F(x)= Id(x) + \delta(x)$, where $\delta: \R \rightarrow \R$, is
a periodic function. It measures the deviation or the distance of a
point $s$ to its image $f(s)$ along one of the many possible chosen
geodesics (counting periodic geodesics).  This choice of angle
(geodesic) is related to the fact that the lifts are not unique. For
any lift $F$, the associated deviation function $\delta$ has bounded
amplitude:
\[
 \max_{x\in \R}\{\delta (x)\}-\min_{x\in \R
} \{ \delta(x)\}<1.
\]
Moreover, the deviation function $\delta \circ p^{-1}: S^1 \rightarrow
\R$ is well defined by periodicity of $\delta$. For a random
homeomorphism $f(\omega, s)$ in $S^1$, given a random lift $F(\omega,
x)$, the corresponding random deviation as a function on $S^1$ will be
denoted, by abuse of notation, by $\delta\circ p^{-1}$, in the sense
that, $\delta\circ p^{-1}(\omega, s)=\delta (\omega, p^{-1}(s))$.


Next, we establish a criterion of \textit{uniformness} on the random lifts:
\begin{definition} \label{Def: alpha_lifts} 
 Given an orientation
    preserving homeomorphism $f$ and $q, \alpha \in \R$, we denote by
    $F_{q, \alpha}: \R \rightarrow \R $ the unique lift of $f$ such
    that $F_{q, \alpha} (q)\in [\alpha, \alpha +1)$. We call this
    function the $(q,\alpha)$-lift of $f$.
\end{definition}
\noindent We shall refer to such choices of parameters fulfilling Definition
\ref{Def: alpha_lifts} as \textit{uniform choice of lifts}. 
The following elementary properties of the deviation periodic function
$\delta$ are direct consequences of Definition \ref{Def: alpha_lifts}.
\medskip
\begin{proposition}[Basic properties] \label{Prop: propriedades da delta }
For all $q, \alpha \in \R$ and $f \in \mathcal{H}^+$, write the 
corresponding lift $F_{q, \alpha}= Id + \delta_{q, 
\alpha}$. Then the deviation function $\delta_{q, 
\alpha}$ satisfies:
\begin{enumerate}
 \item {\em Boundedness}: For every $x \in \R$,
 \begin{equation} 
\label{Prop: propriedade limitacao do delta}
(\alpha- q)-1 < \delta_{q, \alpha}(x)< (\alpha-q) + 2;
\end{equation}
\item {\em  Periodicity}: for all $k,l \in
\Z$,
\begin{equation} \label{Prop: propriedade periodicidade do delta}
\delta_{q + k, \alpha + l} = 
\delta_{q,
  \alpha} + l - k.
  \end{equation}
\end{enumerate}
\end{proposition}
\begin{proof}
  For item (1); note that, directly from the definition of $F_{q,
    \alpha}$, we have that $\alpha -q \leq \delta_{q, \alpha}(q)<
  (\alpha -q)+ 1$. Adding this inequality to the bounded amplitude $-1
  < \delta_{q, \alpha}(x)- \delta_{q, \alpha}(q) < 1$, the result
  follows. Item (2) is obvious.
\end{proof}

We are particularly interested in the dependence of the rotation
number with respect to the parameters $(q, \alpha)$. In contrast to
the deterministic dynamics, in random systems, the rotation number
does depend on $(q, \alpha)$. In order to understand this dependency
and to obtain comparison formulae relating the family of rotation
numbers, it is convenient at this point to consider that the rotation
numbers live in the real line $\R$ rather than in $\R / \Z$.  We
introduce the following definition:

\begin{definition} \label{Def: RN rho_q_alpha} For a choice of
  parameters $q,\alpha \in \R$, the $(q, \alpha)$-rotation number of
  the random dynamical system generated by a cocycle $(f, \theta)$ is
  the random variable defined by the asymptotic limit
\[
 \rho_{q,\alpha}=\lim_{n\rightarrow \infty} \frac{F^{n}(\omega, x) - x}{n},
\]
when the limit exists, where $F^{n}(\omega, x)= 
F_{q,\alpha}(\theta^{n-1 }\omega, \cdot) \circ
F_{q,\alpha}(\theta^{n-2}\omega , \cdot) \circ \cdots \circ  
F_{q,\alpha}(\omega,
x)$.
\end{definition}
\noindent The almost surely existence of $\rho_{q,\alpha}$ is proved
in Theorem \ref{Thm: existencia RN_alpha}. Moreover, if
$\rho_{q,\alpha}$ exists for a certain $\omega \in \Omega$, it is
independent of the starting point $x\in \R$.

\begin{proposition} \label{Prop: inv RN pela CI} If the rotation
  number $\rho_{q, \alpha}$ of Definition~\ref{Def: RN rho_q_alpha}
  exists for a fixed $\omega \in \Omega$ and an initial $x\in \R$,
  then it is independent of $x\in \R$.
\end{proposition}

\begin{proof}
  For every positive integer $n$, we have that $F^{n}(\omega, \cdot)$
  defined as before is again a lift of the composition $f(\theta^{n-1 }\omega, 
\cdot) \circ \cdots \circ  
f(\omega,
\cdot)$
  (although not necessarily with $F^{n}(\omega, q) \in [\alpha,
  \alpha+ 1) $). Hence, its
  deviation $(F^n - Id)$ has bounded amplitude
\[
\max_{x\in \mathbb{R}}\{(F^n-Id)(x)\}-\min_{x\in \mathbb{R}
}\{(F^n-Id)(x)\}<1\ .
\]
Then, for all $x, y \in \R$,
\[
|F^{n}(\omega, x)-F^{n}(\omega, y)|\leq |(F^n (\omega, x)-x)-(F^{n}(\omega, y)-y)|+|x-y|\leq |x-y|+1,
\]
therefore
\[
\lim_{n\rightarrow \infty }\left( \frac{F^{n}(\omega, x)}{n}-\frac{F^{n}(\omega, y)}{n}\right) 
=0.
\]

\end{proof}

Let $\Theta: \Omega \times S^1 \rightarrow \Omega \times S^1$ be the
skew product of the dynamics $(f, \theta)$ given by $\Theta(\omega, s)
= (\theta (\omega), f (\omega, s))$. An invariant probability measure
$\mu$ on $\Omega \times S^1$ factorise uniquely (a.s.) as $\mu ( d
\omega, ds)= \nu_{\omega}(ds)\, \Prob (d \omega)$, where
$\nu_{\omega}$ are random probability measures on $S^1$, see
\textit{e.g.} Arnold \cite[Sec. 1.4 and 1.5]{Arnold} and references
therein.  In particular, if the sequence of random homeomorphisms $(f
(\theta^{n}, \cdot))_{n\geq 0}$ is \textit{i.i.d} with respect to an
appropriate probability space, then $\mu = \nu(ds) \Prob (d \omega)$,
where $\nu(ds)$ is the stationary measure on $S^1$ for the Markov
process induced by $(f (\theta^{n}, \cdot))_{n\geq 0}$ a. s., see
\cite[Sec. 1.4.7]{Arnold}.


\begin{theorem}[Existence of $(q, \alpha)$-rotation numbers] 
\label{Thm: existencia RN_alpha} Let $\mu= \nu_{\omega}(ds) \Prob (d \omega)$ 
be 
an invariant probability 
measure on $\Omega \times S^1$ for the skew product $\Theta(\omega, s)$ 
associated to the cocycle $(f, \theta)$. Given parameters $q, \alpha\in \R$, 
the rotation number $\rho_{q, \alpha}$ exists and 
satisfies:
\begin{equation}
\label{eq.ergodic-formula}
 \rho_{q, \alpha} =  \E \int_{S^1}   \ [ \delta_{q, \alpha}(\omega,  
p^{-1}(s))] 
\ d \nu_{\omega} 
(s)  \ \ \ \ \ \ \ \Prob\mbox{-a.s.}.  
\end{equation}
\end{theorem}
\begin{proof} 
  For a fixed pair of parameters $(q, \alpha)$, we write the deviation
  function $\delta (\omega, x):= \delta_{q, \alpha}(\omega,x)$ and
  $\delta \circ p^{-1} (\omega, s):= \delta_{q, \alpha} (\omega,
  p^{-1} (s))$, for the sake of notation. From Definition \ref{Def: RN
    rho_q_alpha} and induction on $n$ we have that

\begin{displaymath}
\begin{split}
  F^{n} (\omega, x) = x &+\delta (\omega ,x) + \delta \Bigl(\theta
  \omega , x+\delta (\omega
  ,x)\Bigr) \\
  &+ \delta \Bigl( \theta^{2} \omega , x+\delta (\omega ,x)+\delta
  \Bigl(\theta\omega ,x+\delta (\omega ,x)\Bigr)\Bigr) + \cdots\\
\vspace{2cm}&+\delta \Bigl(\theta^{n-1}\omega ,x+\delta (\omega 
,x)+\delta (\theta \omega
,x+\delta (\omega ,x))+\cdots +\delta (\theta^{n-2}\omega, \cdots )\Bigr)\ .
\end{split}
\end{displaymath}
For every positive integer $i$, one has that $F^i(\omega, \cdot)$ is a
lift of $f(\theta^{i-1} (\omega), \cdot) \circ \cdots \circ f(\omega,
\cdot)$, \textit{i.e.},
\begin{eqnarray*}
p \Bigl( x+\delta (\omega ,x)+\cdots +\delta (\theta^{i-1}(\omega), \ldots 
)\Bigr) &= & f(\theta^{i-1} (\omega), 
\cdot) \circ \cdots \circ f(\omega, \cdot) \circ p(x).
\end{eqnarray*}
Additionally, for every deviation function and every $z\in \R$, we
have, by periodicity that $\delta(\cdot, z) = \delta\circ
p^{-1}(\cdot, p(z)) $. Thus, the terms in the expression of
$F^n(\omega, x)$ above can be written as
\begin{eqnarray*}
\delta \Bigl(\theta^{i-1}(\omega), x+\delta (\omega ,x)+\cdots 
+\delta (\theta^{i-2}(\omega), \cdots 
)\Bigr) && \\
 && \\
&\hspace{-10cm} =  &  \hspace{-5cm} \delta \circ p^{-1} 
\Bigl(\theta^{i-1}(\omega), f(\theta^{i-1} 
(\omega), 
\cdot) \circ \cdots \circ f(\omega, \cdot) \circ p(x) \Bigr) \\
 & \hspace{-10cm} = & \hspace{-5cm} \delta \circ p^{-1} \Bigl( 
\Theta^{i-1}(\omega, p(x)) \Bigr).
\end{eqnarray*}
Therefore,
\[
F^n(\omega, x) =x+\sum_{i=1}^{n}\delta \circ p^{-1} \Bigl( \Theta ^{i-1}(\omega 
,p(x)) \Bigr).
\]

We may suppose $\mu$ is an ergodic invariant probability measure for
the skew product on $\Omega \times S^1$. We obviously have that
$\delta \circ p^{-1} \in L^1 (\mu)$.  Then, by the Birkhoff's ergodic
theorem, one has that
\[
\rho_{q, \alpha} = \lim_{n\rightarrow \infty }\
\frac{F^n(\omega,x) - x}{n} = {\rm I\!E}\left[ \ \int_{S^{1}}\delta \circ
p^{-1} (\omega ,s)\ d\nu
_{\omega }(s)\ \right] \ \ \ \ \ \mu \mbox{-a.s.}.
\]
Proposition \ref{Prop: inv RN pela CI} tells us that the rotation
number $\rho_{q, \alpha}$ above is independent of the initial point in
$S^{1}$. Hence, the subset of $\Omega \times S^1$ where the equality
above holds extends to the whole vertical fibre, i.e., it is a
Cartesian product $\Omega'\times S^1$, for some $\Omega'\subset
\Omega$. Since $(\theta, \Prob)$ is ergodic, $\Omega'$ has full
$\Prob$-probability measure. Therefore, if there exist more than one
ergodic component in the ergodic decomposition of invariant measures,
the extension to the whole vertical fibre stated by Proposition
\ref{Prop: inv RN pela CI} implies that the rotation numbers $
\rho_{q, \alpha}$ corresponding to each ergodic component must
coincide.

\end{proof}

\noindent Note that the boundedness of the deviations $\delta_{q,
  \alpha}$ for the $(q, \alpha)$-lifts stated by inequality
(\ref{Prop: propriedade limitacao do delta}) implies that $ (\alpha
-q) - 1 < \rho_{q, \alpha} < (\alpha-q) + 2$, and inequality
(\ref{Prop: propriedade periodicidade do delta}) implies that
$\rho_{q+ k, \alpha+ l}=\rho_{q, \alpha} + (l-k)$.

\bigskip

In contrast to the uniform choice of lifts as established in Definition 
\ref{Def: alpha_lifts}, consider a nonuniform choice of random lifts 
$F(\omega)$. Then, 
for fixed parameters $(q,\alpha)$, we have  $F(\omega) = F_{q,\alpha} + 
N(\omega)$, where $N(\omega)$ is an integer random variable.

\begin{corollary} \label{Cor: existence with alternative lifts} If we
  consider a nonuniform choice of random lift given by $F(\omega) =
  F_{q,\alpha} + N(\omega)$, where $N(\omega)$ is an integrable
  integer random variable, then the associated rotation number for
  this lift exists and it is given by

\[
\rho_{F(\omega)}=  \lim_{n\rightarrow \infty} \frac{F^n (x) - x}{n} = 
\rho_{q,\alpha} + \E[ N].
\]
\end{corollary}

\begin{proof}
  The proof of Theorem \ref{Thm: existencia RN_alpha} holds for any
  random lift $F(\omega)$.  Given $F(\omega) = F_{q,\alpha} +
  N(\omega)$, with $N(\omega)$ integrable, we have that the deviation
  function for this lift is $\delta= \delta_{q,\alpha}(\omega, x) +
  N(\omega)$. Hence the statement follows directly from the ergodic
  formula (\ref{eq.ergodic-formula}) of Theorem \ref{Thm: existencia
    RN_alpha}.

\end{proof}

Last Corollary says, in particular, that if the choice of the lifts
are nonuniform, the associated ``rotation number'' can be any real
number, without any well defined dynamical meaning. Here we point out
a criticism on the approach of \cite{Li and Lu}, where the lifts are
arbitrary (integrable), as in the statement of the Corollary \ref{Cor:
  existence with alternative lifts}. In fact, this Corollary
together with Proposition~\ref{Prop: compara RN topologicos} ahead say
that there is no canonical choice for the parameters $(q, \alpha)$,
hence there is no canonical choice of rotation number. In
Section~\ref{sec.samplingThm}, (see Remark~\ref{remark-end}) we are
going to show that the lifts close to the identity in $\R$, precisely,
with $q-1<\alpha <q$, are somehow the right choices.

\medskip

We can now obtain a probabilistic formula of comparison between
distinct rotation numbers $\rho_{q, \alpha}$ and $\rho_{q', \alpha'}$
in terms of their parameters. Before doing that, it is convenient to
set the following notation.
For the sets of parameters $(q, \alpha)$ and $(q', \alpha')$, and $i,j
\in \Z$, consider
\begin{displaymath}
\begin{split}
  A_{i} :=& \{ \omega \in \Omega : F_{q, \alpha'}(\omega, q') \in [\alpha'+ i, \alpha'+ i + 1)\}, \text{ and} \\
  B_{j} :=& \{ \omega \in \Omega : F_{q, \alpha}(\omega, q) \in [\alpha'- j, \alpha'-j + 1)\}.
\end{split}
\end{displaymath}
For an interpretation of the disjoint sets $A_i$, we represent in
Figure~\ref{fig1a} a fragment of the graphic of $F_{q, \alpha'}(x)$ for two
distinct hypothetical $\omega_1$ and $\omega_2$ for $q'$, say, in a
neighbourhood of $[q, q+ 1]$. For $q'$ in other subintervals, the
interpretation is just an extension of what happens in a neighbourhood
of $[q, q+1]$. For $b-1<q'<a$ we have that $\omega_1, \omega_2 \in
A_0$. For $a<q'<b$ we have that $\omega_1 \in A_1$, and $\omega_2 \in
A_0$. And for $b<q'<a+1$ we have that $\omega_1, \omega_2 \in
A_1$. The picture also shows that the $A_i$'s are all empty, except
when $i\in \{\lfloor q'- q\rfloor, \lfloor q'- q\rfloor + 1
\}$. Observe that if $q'=q+n$, for an integer $n$ then $A_n= \Omega$
is the only nonempty set in the family $(A_i)_{i\in \Z}$.
Analogously, for an interpretation of the sets $B_j$, we represent in
Figure~\ref{fig1b} a fragment of the graphic of $F_{q, \alpha}$ for two
distinct $\omega_1$ and $\omega_2$ in a neighbourhood of $q$.  We are
interested on the dependence of the sets $B_j$ on the parameter
$\alpha'$, running in the vertical axes.  For $d-1<\alpha'<c$ we have
that $\omega_1, \omega_2 \in B_0$. For $c<\alpha'<d$ we have that
$\omega_1 \in B_1$, $\omega_2 \in B_{0}$.  And for $d<\alpha'<c+1$ we
have that $\omega_1, \omega_2 \in B_1$. From the picture, one also
sees that the $B_j$ are all empty, except when $j\in \{\lfloor
\alpha'- \alpha\rfloor, \lfloor \alpha'- \alpha\rfloor + 1
\}$. Observe that if $\alpha'=\alpha+n$, for an integer $n$ then $B_n=
\Omega$ is the only nonempty set in the family $(B_j)_{j\in \Z}$.

 Let us define
\begin{displaymath}
\begin{split}
  k =& \min \left\{ i \in \Z, \mbox{ such that } \Prob \left( A_ {i} \right)> 0 
\right\},  \text{ and} \\
l =& \min \left\{ j \in \Z, \mbox{ such that } \Prob \left( B_{j} \right)> 0 
\right\}.
\end{split}
\end{displaymath}

\begin{figure}[t]
\label{Fig1}
    \subfigure[]{\label{fig1a}\includegraphics[width=.55\columnwidth]{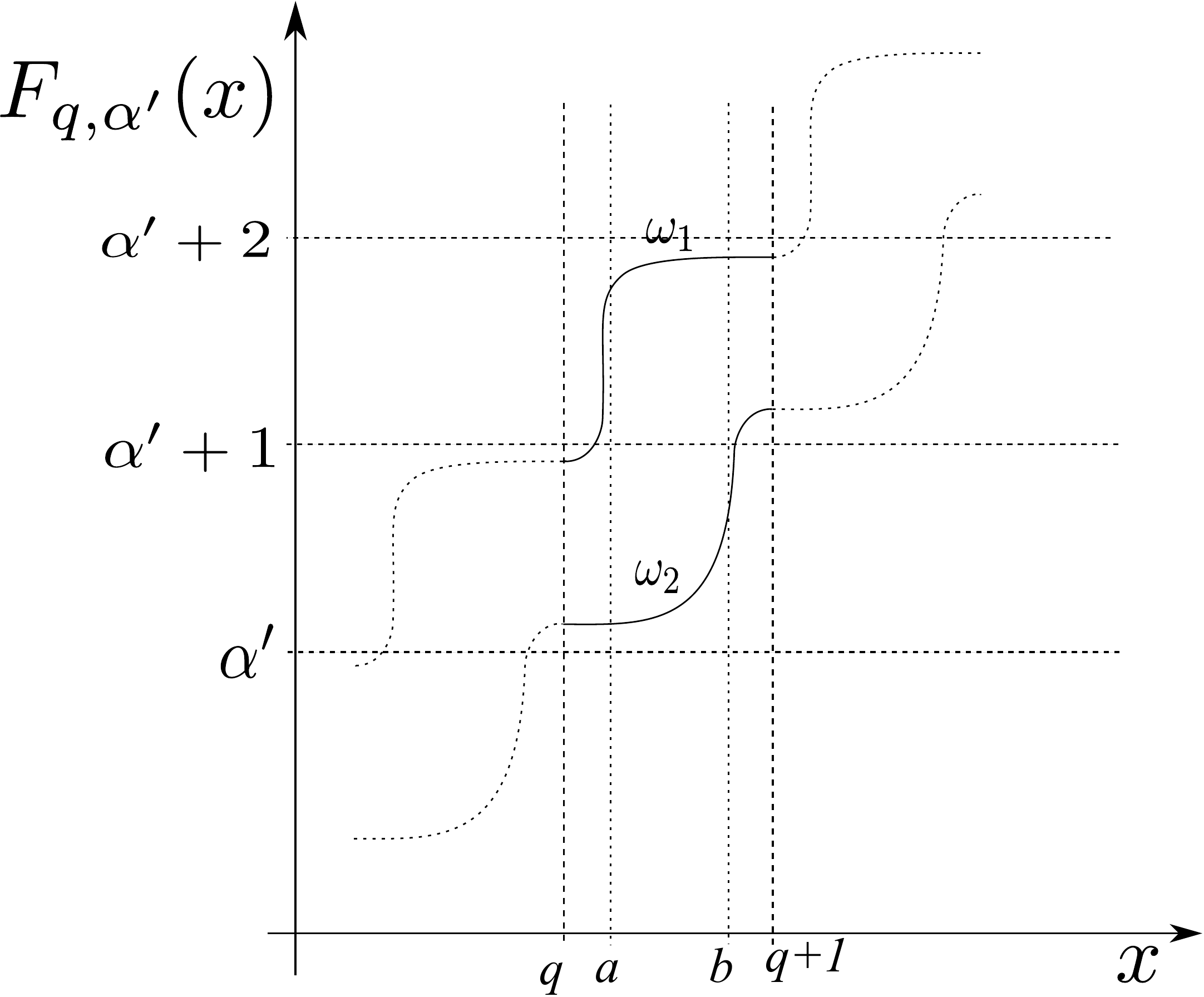}}
    \subfigure[]{\label{fig1b}\includegraphics[width=.4\columnwidth]{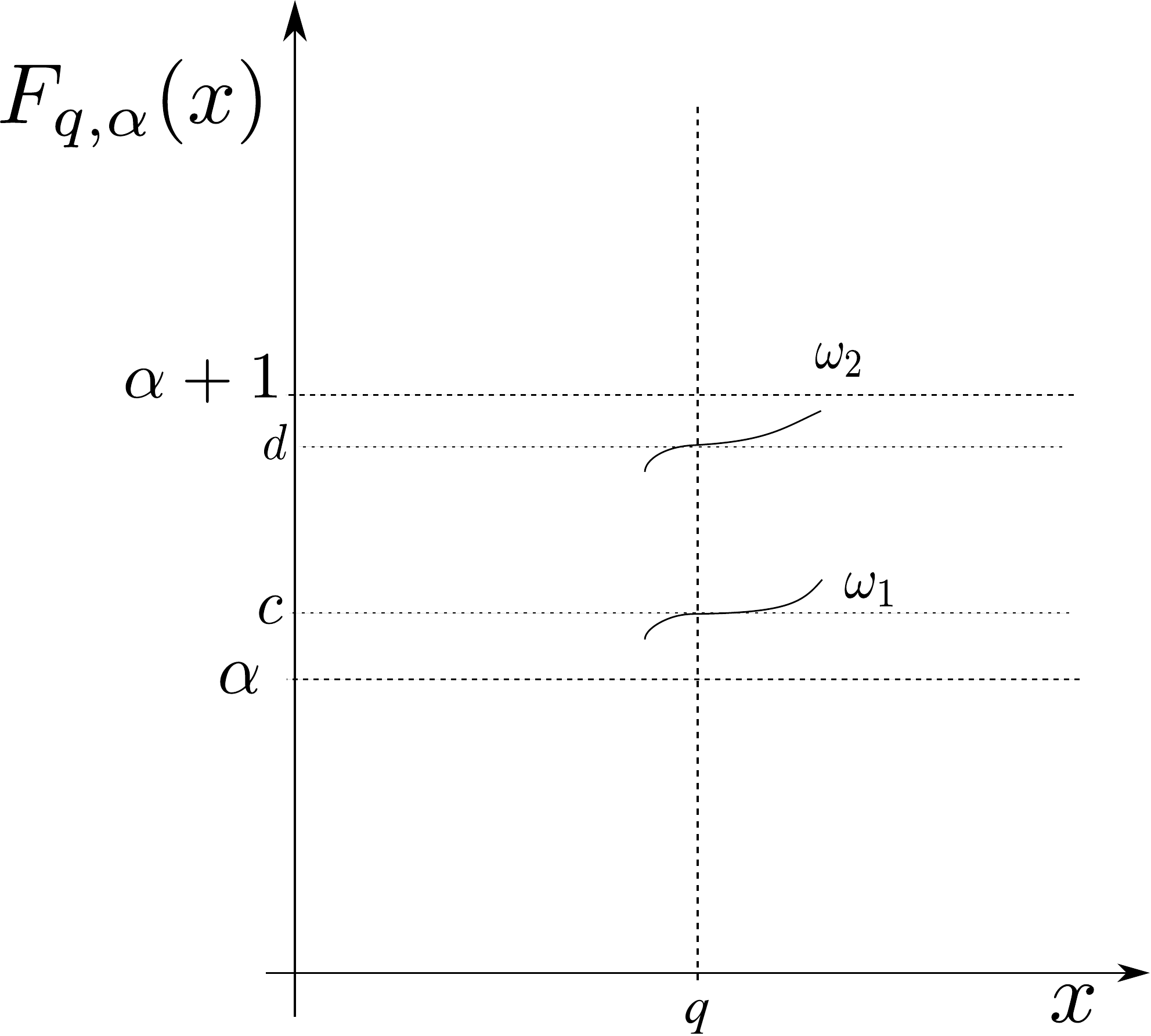}}
  
\vspace*{-0.3cm}
  \caption{(a) Interpretation of the sets $A_i$. (b) Interpretation for sets $B_j$.}
  \label{fig.FIG1}
\end{figure}

Note that $k$ above, as a function of $q'$, is right-continuous,
non-decreasing, constant by parts, with unitary jumps at the points of
a sequence $(q'_n)_{n\in \Z}$ with unitary increments, such that
$k(q'_n)=n$ and $q'_0\leq q<q'_1$. Moreover, the probability $\Prob
(A_{k(q')})$ is a right-continuous periodic function depending on
$q'$, with unitary positive jump discontinuities at the points of the
sequence $(q'_n)_{n\in \Z}$. It is decreasing in each interval $[q'_n,
q'_{n+1})$, with $\Prob (A_{k(q')})= 1$ if $q'= q'_n$ and decreases to
zero when $q'$ approaches $q'_{n+1}$ on the left.  Analogously, $l$ as
a function of $\alpha'$, is right-continuous, non-decreasing, constant
by parts, with unitary jumps at the points of a sequence
$(\alpha'_n)_{n\in \Z}$ with unitary increments, such that
$l(\alpha'_n)=n$ and $\alpha'_0\leq q<\alpha'_1$. Moreover, the
probability $\Prob (B_{l(\alpha')})$ is a right-continuous periodic
function depending on $\alpha'$, with unitary positive jump
discontinuities at the points of the sequence $(\alpha'_n)_{n\in
  \Z}$. It is decreasing in each interval $[\alpha'_n,
\alpha'_{n+1})$, with $\Prob (B_{l(\alpha')})= 1$ if $\alpha'=
\alpha'_n$ and decreases to zero when $q'$ approaches $\alpha'_{n+1}$
on the left.

We can state now the following relation between the rotation numbers 
$\rho_{q, \alpha}$, and $\rho_{q', \alpha'}$.
\begin{proposition} \label{Prop: compara RN topologicos} Let $(q, \alpha)$ and  
$(q', \alpha') $ be two sets of 
parameters for the rotation number of Definition (\ref{Def: RN rho_q_alpha}). 
Then we have that
\[
 \rho_{q', \alpha'} = \rho_{q, \alpha}  + \Prob (A_{k(q')}) - \Prob 
(B_{l(\alpha')})  - k + l. 
\] 
\end{proposition}
\begin{proof}
  By the definition of the uniform random lift $F_{q, \alpha'}(\omega,
  \cdot)$ we have that $\Prob (A_{k} \dot{\cup} A_ {k+1}
  )=1$. Therefore,
 \begin{eqnarray*}
  F_{q', \alpha'} (\omega, x) & = & F_{q, \alpha'}(\omega, x) - k\, 
1_{A_k}(\omega) - 
(k+1)\, 1_{A_{k+1}}(\omega)\\
 & =& F_{q, \alpha'}(\omega, x) + 1_{A_k}(\omega) - k - 1,
 \end{eqnarray*}
 and thus, the deviations $\delta_{q', \alpha'} = \delta_{q, \alpha'}+
 1_{A_k}-k -1 $. This implies, by Theorem \ref{Thm: existencia
   RN_alpha}, in particular that
 \begin{equation} \label{form: comparacao 1}
 \rho_{q', \alpha'} = \rho_{q, \alpha'} + \Prob (A_k) -k - 1.
 \end{equation}
 
 \bigskip

 Again, by the uniformness of the lifts $F_{q, \alpha}$ we have that
 $\Prob (B_l \dot{\cup} B_ {l+1} )=1$, and,
 \begin{eqnarray*}
  F_{q, \alpha'}(\omega, x) &= & F_{q, \alpha}(\omega, x) + l \, 1_ { \{ B_l \} 
}(\omega) + (l+1) \, 1_{\{ B_{l+1}\}}(\omega) \\
 & = & F_{q, \alpha}(\omega, x) - 1_{B_l}(\omega)+ l + 1.
 \end{eqnarray*}
Hence, the deviations $\delta_{q, \alpha'} = \delta_{q, \alpha}- 1_{B_l} + l + 
1$. This implies, by Theorem \ref{Thm: existencia RN_alpha}, that 
  \begin{equation} \label{form: comparacao 2}
 \rho_{q, \alpha'} = \rho_{q, \alpha} - \Prob (B_l) + l + 1.
 \end{equation}
The result now follows by Formulae (\ref{form: comparacao 1}) and  (\ref{form: 
comparacao 2}) above.

\end{proof}

From Proposition \ref{Prop: compara RN topologicos} and the comments
right before it, we have that, for a fixed $\alpha$, the rotation
number $\rho_{q, \alpha}$ is right-continuous and decreasing with
respect to $q$. And for a fixed $q$, the rotation number $\rho_{q,
  \alpha}$ is right-continuous and increasing with respect to
$\alpha$.  Moreover, discontinuities with respect to $(q, \alpha)$
happens if and only if there exist an atom in the distribution of the
random variable $F_{q, \alpha}(\omega, q')$ for a certain $q'$. If
instead the random variable $F_{q, \alpha}(\omega, q')$ is absolutely
continuous with respect to Lebesgue measure in $\R$ for all $q' \in
\R$ (actually for $q'\in [q, q+1)$ is enough), then $\rho_{q, \alpha}$
is continuous in $(q, \alpha)$.

Li and Lu, in \cite[Thm. A.ii]{Li and Lu}, have proved that if a
random lift $F(\omega)$ is integrable, then, the map $L^1(\Omega)
\mapsto \rho$, which sends lifts $F(\omega)$ into its associated
rotation number, is continuous. In our case, if the distribution of
the random variable $F_{q, \alpha}(\omega, q')$ does have atoms, it
generates discontinuities for the rotation number $\rho_{q, \alpha}$,
which does not contradict the continuity result of Li and Lu since, in
this case, $(q, \alpha) \mapsto L^1(\Omega)$ is not continuous.

\section{Rotation number of orbits}

In this section we explore the concept of rotation number based on the
physical observation of the angular behaviour of single
orbits. Fix an $s_0 \in S^1$ and consider its random trajectory $s_n = 
f(\theta^{n-1}(\omega),
\cdot) \circ \cdots \circ f(\omega, s_0) $. We are going to lift this
orbit to $\R $ as an increasing random angle with bounded jumps. The
asymptotic average speed in $\R$ is the rotation number of its orbit. More
precisely, let $\gamma_0 \in [0, 1)$ be the initial normalised angle,
\textit{i.e.} $p(\gamma_0)= s_0$. Defining by induction $\gamma_n \in
\R$ as the angle in $\R$, such that, $p(\gamma_n)= s_n$ and
$\gamma_{n-1} \leq \gamma_n < \gamma_{n-1} + 1 $, gives us the
following.
\begin{definition} The rotation number of the random orbit $s_n$
  starting at $s_0$ is defined by the random variable
 \[
  OR_{(s_0)} = \lim_{n\rightarrow \infty} \frac{\gamma_n - \gamma_0}{n},
 \]
 when the limit exists.
\end{definition}
Next, we show that the rotation number of orbits exists almost surely
with respect to ergodic invariant measures for the skew product on
$\Omega \times S^1$.  Differently from the topological rotation number
$\rho_{q, \alpha}$ of the previous section, the rotation number of
orbits can in fact depend on the initial condition, c.f. Example 1 below.

\bigskip

\begin{theorem} \label{Thm: existencia NR de orbita} Let $\mu$ be an
  ergodic invariant probability measure for the skew product $\Theta$
  on $\Omega \times S^1$. Then the rotation number of an orbit $(s_n)$
  exists for $\mu$-almost every $(\omega, s_0)$.

\end{theorem}

\begin{proof}
  Let $\bar{\delta}(x)$ be the unique (possibly discontinuous)
  deviation function with image in $[0,1)$, such that, $\bar{F}(x) =
  Id(x) + \bar{\delta}(x)$ satisfies $f(\omega, \cdot ) \circ p = p
  \circ \bar{F}(\omega, x)$, \textit{i.e.}, $\bar{F}$ is a
  discontinuous lift of the homeomorphism $f$. The dynamics of the
  increasing angles $\gamma_n$ can be described in terms of this
  discontinuous lift $\bar{F}$ as follows:
\[
 \gamma_n = \bar{F}(\theta^{n-1}(\omega), \cdot) \circ \cdots \circ 
\bar{F}(\omega, 
\gamma_0).
\]
The rest of the proof follows as in the proof of Theorem \ref{Thm:
  existencia RN_alpha}, since the same calculation there also holds
for discontinuous lifts. We have obviously that $\bar{\delta} \in L^1
(\mu)$, hence,
\[
 OR_{(s_0)} = \E \int_{S^1} \bar{\delta}(\omega, x) \  \nu_{\omega} (ds) \ \ 
\ \
\ \ \ \mu \mbox{- a.s.}
\]

\end{proof}

Next Corollary gives a comparison formula between the orbit rotation
number $OR_{s_0}$ of this section and the family of topological
rotation numbers $\rho_{q, \alpha}$ of the last section.

\begin{corollary} \label{Prop: compara RN fisico com topologico} Fix a
  pair of parameters $(q, \alpha)$.  For each ergodic invariant
  probability measure $\mu(d\omega, ds)= \nu_{\omega}(ds)\Prob (d
  \omega)$ on $\Omega \times S^1$ we have that
\begin{eqnarray*}
OR_{(s_0)}  =  &\rho_{q,\alpha} & -   k \ \ \E \,[ \, \mu_{\omega} \{s\in S^1: 
\delta_{q,\alpha}(p^{-1}(s)) \in (k, k+1) \}] \\
 && -   (k+1) \ \E \, [\, \mu_{\omega} \{s\in S^1: 
\delta_{q,\alpha}(p^{-1}(s))\in [k+1, k+2) \}]\\
& & - (k+2)  \ \E \,  [\, \mu_{\omega} \{s\in S^1: 
\delta_{q,\alpha}(p^{-1}(s))\in [k+2, k+3) \}]
\end{eqnarray*}
$\mu$-a.s., where $k= \lfloor (\alpha-q) -1
\rfloor$.
\end{corollary}
\begin{proof} We write the discontinuous deviation
  $\bar{\delta}(\omega, x) \in [0,1)$ associated to $OR_{(s_0)}$ in
  terms of the uniform deviation $\delta_{q,\alpha}$. In order to do
  that, we only have to add to $\delta_{q,\alpha}$ appropriate
  integers which depend on $\omega$ and on $x\in \R$. Inequality
  (\ref{Prop: propriedade limitacao do delta}) says that $(\alpha-
  q)-1 < \delta_{q, \alpha}(x)< (\alpha-q) + 2$, which implies that
  this correction random integer ranges only in the set $\{ k, k+1,
  k+2\}$. In particular, one checks that
\begin{eqnarray*}
  \bar{\delta}(\omega, x)= 
  \delta_{q,\alpha}(x) &-& k\  1_{\{ \delta_{q,\alpha} (\omega,x) \in (k, 
    k+1)\} } \\
  &-& (k+1)\, 1_{\{ \delta_{q,\alpha} (\omega,x) \in [k+1, k+2)\}} \\
  &-& (k+2)\, 1_{\{ 
    \delta_{q,\alpha} (\omega,x) \in [k+2, k+3)\}}.
 \end{eqnarray*}
 The result follows by the ergodic formula at the end of the proof of
 Theorem \ref{Thm: existencia NR de orbita}.
 
\end{proof}
As an example, one verifies that with parameters $(q, \alpha) =
(0,0)$, we have that $OR_{(s_0)} = \rho_{0,0} + \E [\mu_{\omega} \{s\in S^1:
\delta_{0,0}(p^{-1}(s))< 0 \}] - \E [\mu_{\omega} \{s\in S^1:
\delta_{0,0}(p^{-1}(s))\geq 1 \}]$, $\mu$-a.s..

\bigskip

Next example is a simple construction which shows the dependency of
the orbit rotation number with the initial condition of the skew
product $\Theta$.

\bigskip

\begin{example} \emph{
  Parametrise $S^1$ by $x\mapsto e^{2\pi x i}$ with $x\in (-1/2,
  1/2]$. Let $\Omega = \{ 1,2,3,4 \} $, with $\Prob [i]= \frac{1}{4}$
  for every $i \in \Omega$. Consider the ergodic transformation given
  by the permutation which in cycle notation is $\theta=
  (1,2,3,4)$. Let $f(1),f(2),f(3), f(4): S^1 \rightarrow S^1$ be
  homeomorphisms in $\mathcal{H}^+$ such that $f(\omega,0)=0$, for all
  $\omega \in \Omega$ and $f(1, 1/8)=3/8$, $f(2, 3/8)=-3/8$, $
  f(3,-3/8)=-1/8$ and $f(4,-1/8)=1/8$.  For this random dynamics we
  have that $\rho_{0, \alpha}= \lfloor \alpha \rfloor $, since zero is
  a fixed point and $\delta_{0, \alpha}(\omega, 0)= \lfloor \alpha
  \rfloor $ for all $\omega \in \Omega$. The orbit rotation number
  starting at zero is given by $OR_{0}=0$, corresponding to a Dirac
  invariant measure at zero for all $\omega \in \Omega$. Nevertheless,
  for $x_0= 1/8$, $\omega= 1$ we have that the evolution of increasing
  angles is given by
\[
\gamma_{n}=\frac{2n+1}{8},
\]
which yields $OR_{\frac{1}{8}}= 1/4$, corresponding to the ergodic
invariant measure $\mu$ given by the normalised sum of Dirac measures
in $(1, \frac{1}{8})$, $(2, \frac{3}{8})$, $(3, -\frac{3}{8})$ and
$(4, -\frac{1}{8})$, \textit{i.e.} a periodic orbit of $\Theta$.}
\end{example}

\section{Sampling time Theorem}
\label{sec.samplingThm}

In \cite{Ruffino_ROSE} it has been studied the rotation number for a
sequence of random matrices acting on $\R^2$. Such quantity
corresponds to a counter part of the Oseledet's theorem for Lyapunov
exponents for a product of random matrices with invariant 2-subspaces,
see \textit{e.g.}, Arnold \cite{Arnold} and references therein. In the
context of linear systems on $S^1$, it was proved in that paper a
sampling theorem for discretisation of the flow at time interval
$\Delta t$, such that, the rotation number of the discrete system,
when rescaled by $\frac{1}{\Delta t}$, converges to the rotation
number of the original continuous systems. The main result in this
section generalises this sampling theorem, extending its scope from
linear systems acting on $S^1$ to nonlinear equations intrinsic on
$S^1$. We are going to show that, adequate choices of the parameters
$(q, \alpha)$ will lead to a compatibility of the definitions of
rotation numbers for discrete systems with continuous systems.

Consider a classical stochastic flow of
diffeomorphisms on $S^1$ generated by a Stratonovich differential
equation:
\begin{equation} 
\label{Eq: fluxo estocastico em S^1} 
ds_t = H^0 (s_t)\ dt + \sum^{m}_{j=1} H^j(s_t)\ \circ dB_t^j
\end{equation}
with initial condition $s_0 \in S^1$, where $H^0, H^1, \ldots , H^m$
are smooth vector fields on $S^1$ and $(B^1_t, \ldots, B_t^m)$ is a
standard Brownian motion on $\R^m$ with respect to a filtered
probability space $(\Omega, \mathcal{F}, \mathcal{F}_t, \Prob)$. Here we 
consider the classical Wiener space 
$\Omega= \{ \sigma:[0,\infty] \rightarrow \R, \mbox{ continuous, with } 
\sigma(0)=0 
\}$ with the Wiener probability measure, such that the canonical process $t 
\mapsto \omega(t)$ is a Brownian motion. The ergodic shift $\theta_t: \Omega 
\rightarrow \Omega$ is given by $\theta_t (\omega)(s)= \omega(t+s)-\omega(t)$, 
for all $t\geq 0$. 
Equation (\ref{Eq: fluxo estocastico em S^1}) can be lifted to the covering 
space $\R$, and written as
\begin{equation} 
\label{Eq: flow estocastico em R}
 dx_t = h^0 (x_t)\ dt + \sum^{m}_{j=1} h^j(x_t)\ \circ dB_t^j
\end{equation}
with initial condition $x_0$ such that $p (x_0)= s_0$, where the
functions $h^i$, $i=0,1,\ldots, m$ are periodic with $h^i (x)= H^i
(p(x))$ for all $x\in \R$. If $x_t$ is a solution of equation
(\ref{Eq: flow estocastico em R}), then $s_t= p(x_t)$ solves Equation
(\ref{Eq: fluxo estocastico em S^1}). Let $\varphi_t$ denote the
stochastic flows of diffeomorphisms on $S^1$ generated by Equations
(\ref{Eq: fluxo estocastico em S^1}), and let $\psi_t$
denote the stochastic flows of diffeomorphisms on $\R$ generated by
Equations (\ref{Eq: flow estocastico em R}). Then $p \circ \psi_t =
\varphi_t \circ p$, \textit{i.e.} $\psi_t$ are Poincar\'e lifts of
$\varphi_t$ for all $t\geq 0$. The classical rotation number of the
continuous stochastic flow $\varphi_t:S^1 \rightarrow S^1$ is given by
the average winding of trajectories:
\begin{equation*}
 \mathrm{rot}(\varphi) = \lim_{t\rightarrow \infty } \frac{\psi_t 
(\omega, x_0)}{t}
\end{equation*}
whose existence $\Prob$-almost surely is guaranteed by standard
stochastic analysis technique and the ergodic theorem for Markov
process, see \textit{e.g.} \cite{Ruffino_SSR}. This rotation number is
independent of $x_0$ and is given by
\[
 \mathrm{rot}(\varphi) = \int_{S^1} H^0(s) + \frac{1}{2} \sum_{j=1}^m \left( 
dH^i(s) H^i(s) \right)\ d\nu (s) \ \ \ \ \ \ \Prob\mbox{-a.s.},
\]
where $\nu$ is an ergodic invariant measure on $S^1$ for the diffusion 
generated  by equation (\ref{Eq: fluxo estocastico em S^1}).
 
Next theorem says that if the parameters $q$ and $ \alpha$ are in an
appropriate range, when we sample the time periodically at an interval
$ \Delta t>0$, then the normalised rotation number $\rho_{q,\alpha}$
of the discrete cocycle of diffeomorphisms $\varphi_{\Delta t}$
recover the original (continuous) rotation number
$\mathrm{rot}(\varphi)$, when the frequency of sampling tends to infinite.

\begin{theorem} \label{thm: amostragem} For $q-1 <\alpha < q \in \R$,
  the rotation number of a stochastic continuous systems (\ref{Eq:
    fluxo estocastico em S^1}) can be recovered by the limit of the
  rescaled rotation:
 \[  
 \mathrm{rot} (\varphi_t) = \lim_{\Delta t \rightarrow 0}
 \frac{\rho_{q, \alpha} (\varphi_{\Delta t})}{\Delta t},
 \]
 where $\rho_{q, \alpha} (\varphi_{\Delta t})$ is the topological
 rotation number of the discrete random system generated by
 $\varphi_{\Delta t}(\omega, \cdot)$ with respect to the 
canonical shift $\theta=\theta_{\Delta t}$ in the probability space $\Omega$.
\end{theorem}
For proving \ref{thm: amostragem}, we exploit the fact that at each
time $t\geq 0$, the flow $\psi_t(\omega, \cdot)$ for the system on
$\R$, Equation (\ref{Eq: flow estocastico em R}), is a lift of the
system $\varphi_t(\omega, \cdot)$ on $S^1$, Equation (\ref{Eq: fluxo
  estocastico em S^1}). We have to control, for $t>0$, those lifts
$\psi_t= Id + \delta_t$ whose corresponding deviation $\delta_t$ are
not in the prescribed interval $\delta_t (q) \in [\alpha, \alpha +1)
$.

\begin{proof}[Proof of Theorem~\ref{thm: amostragem}]
  Let $F_{\Delta t}(\omega):\R \rightarrow \R $ denote the $(q,
  \alpha)$-lift of $\varphi_t$ according to Definition \ref{Def:
    alpha_lifts}. Then $F_{\Delta t} = \psi_{\Delta t} + N_{\Delta
    t}(\omega)$, where $N_{\Delta t}(\omega)$ is a random integer
  variable. By the cocycle property and Corollary \ref{Cor: existence
    with alternative lifts}, we have that for $n\in \N$,
\begin{eqnarray*}
  \mathrm{rot}(\varphi) & = &\lim_{n\rightarrow \infty} \frac{\psi_{n \Delta 
      t  }(x_0) - x_0}{n 
    \Delta t}  \\
  & = &  \lim_{n\rightarrow \infty} \frac{\psi_{ \Delta t}(\theta^{n-1}\omega, 
    \cdot) \circ \cdots \circ \psi_{\Delta t}(\theta \omega, x_0 )
    \circ \psi_{\Delta t}(\omega, x_0 )}{n 
    \Delta t}\\
  & = & \frac{1}{\Delta t} \left( \rho_{q, \alpha}(\varphi_{\Delta t}) + \E [N] 
  \right).
\end{eqnarray*}
Therefore, we only have to prove that
\[
 \lim_{\Delta t \searrow 0} \frac{\E [|N| ]}{\Delta t} = 0.
\]
Note that the random variable $N$ counts how many times continuous
trajectories of the original system on $S^1$ cross, in the
anticlockwise direction, the point $p (\alpha) \in S^1$ up to time
$\Delta t$. Thus, we have that

\[
 N(\omega)= \sum_{n\in \Z} n 1_{ \Omega_n}  ~,
 \]
 where
 \[
 \Omega_n = \left\{ \psi_{\Delta t}(\omega, q) 
\in [\alpha+n, \alpha + n+1) \right\}.
\]
We control the expectation of $N$ using boundedness on the
distribution of $\psi_{\Delta t} (\omega,q)$. Let $p(t, x,y)$ be the
density of the transition probability measure associated to a
non-degenerate diffusion given by Equation (\ref{Eq: flow estocastico
  em R}). Then, there exists a constant $M>0 $ such that,
\[
\frac{1}{M \sqrt{t}} e^{-M\frac{(x-y)^2}{t}} \leq p (t, x,y) \leq \frac{M}{ 
\sqrt{t}} e^{-\frac{(x-y)^2}{Mt}}.
\]
See Kusuoka and Stroock \cite{Kusuoka and Stroock, Kusuoka and Stroock-1988}. 
Let $N^+= \max \{N, 0\}$ and $N^- =  \max \{-N, 0 \} $, such that 
$N=N^+ - N^-$. Hence, for the positive 
part $N^+$
\begin{eqnarray*}
 \E [ N^+ ] & \leq & M \int_{\alpha +1}^{\infty} \left( \lfloor x-(\alpha+1) 
\rfloor + 1 \right) \frac{1}{\sqrt{\Delta t}} \exp\left\{- \frac{ 
(x-q)^2}{M \Delta t}\right\} \ dx \\
  && \\
  & \leq & M \int_{\alpha +1}^{\infty} \left( x- \alpha \right) 
\frac{1}{\sqrt{\Delta t}} \exp\left\{- \frac{ 
(x-q)^2}{M \Delta t}\right\} \ dx.
\end{eqnarray*}
And for the negative part:
\begin{eqnarray*}
 \E [ N^- ] & \leq & M \int_{-\infty}^{\alpha} \left( \lfloor \alpha -x 
\rfloor + 1 \right) \frac{1}{\sqrt{\Delta t}} \exp\left\{ - \frac{ 
(x-q)^2}{M \Delta t}\right\} \ dx \\
  && \\
  & \leq & M \int_{-\infty}^{\alpha} \left(  \alpha -x + 1 \right) 
\frac{1}{\sqrt{\Delta t}} \exp\left\{ - \frac{ 
(x-q)^2}{M \Delta t}\right\} \ dx.
\end{eqnarray*}
Changing variables, for $\Delta t \in (0,1)$ we have that
\begin{eqnarray*}
 \E [ N^+ ] & \leq & M^{\frac{3}{2}} \int_{\frac{\alpha -q 
+1}{\sqrt{M \Delta t}}}^{\infty} \left( \sqrt{M} u + q - \alpha  \right) 
 \exp\left\{- u^2 \right\} \ du
\end{eqnarray*}
and
\begin{eqnarray*}
 \E [ N^- ] & \leq & M^{\frac{3}{2}} \int_{\frac{ q -\alpha}{\sqrt{M \Delta 
t}}}^{\infty} \left( \sqrt{M} u + \alpha -q + 1  \right) 
 \exp\left\{- u^2 \right\} \ du
\end{eqnarray*}

Hence $\E [ N^+ ]$ and $\E [ N^- ]$ goes to zero when $\Delta t$ goes to zero. 
Moreover, by standard calculus argument, 
using that $\lim_{z\rightarrow 0} 
\exp\left\{ {-\frac{1}{z}} \right\} z^{\beta}=0$ for any exponent $\beta \in 
\R$, then finally we 
get that
\[
  \lim_{\Delta t \searrow 0} \frac{\E [N^+ ]}{\Delta t} =  \lim_{\Delta t 
\searrow 0} \frac{\E [N^- ]}{\Delta t} = 0.
\]

\end{proof}

\begin{remark}
\label{remark-end}
\emph{ We remark that this sampling theorem does not hold either if $\alpha 
\geq q$ 
or 
if  $\alpha \leq  q-1 $. In this case, note that either $\E [ N^+ ]$ or $\E [ 
N^- ]$ 
does not converge to zero when $\Delta t$ goes to zero. Also, the 
theorem does not hold if one consider the 
rotation number of orbits (Section 3). As a counterexample  
consider a
deterministic system driven by a North-South vector fields in $S^1$ embedded 
in $\R^2$. The rotation number $\mathrm{rot}(\varphi)$
 of the continuous system is obviously zero. But the orbit rotation number 
$OR_{s_0}$ for any discretisation in time  of the flow is zero for orbits 
with $s_0$ on the right hand side of $S^{1}$, and 1 for
orbits with $s_0$ on the left hand side of $S^1$, independently of the time 
interval of discretisation $\Delta t$. }
\end{remark}

\section*{Acknowledgements}

The authors would like to thank the referee for valuable comments and 
suggestions leading to improvements in the article. C.S.R. acknowledges the 
hospitality in the Institute of Mathematics of
the Universidade Estadual de Campinas during his visit in
2010/2011. He has been partially supported by the European Research
Council under the European Union's Seventh Framework Program
(FP7/2007-2013) / ERC grant agreement $n^o$ 267087. P. R. acknowledges
the hospitality in the Max Planck Institute for Mathematics in
Sciences - Leipzig during visits in 2011 and 2013. He has been
partially supported by FAPESP 11/50151-0, 12/18780-0 and CNPq 477861/2013-0.




\bibliographystyle{amsalpha}

\end{document}